\documentclass[a4paper,10pt]{amsart}
\usepackage[utf8]{inputenc}
\usepackage{tikz-cd}
\usepackage{mathtools}
\usepackage{amsbsy}
\usepackage{amssymb,latexsym} 
\usepackage{amsmath,amsthm} 
\usepackage[russian,english,]{babel}
\usetikzlibrary{cd}
 \usepackage{mathtools}
 \usepackage{pifont}
\usepackage{calligra}
 \usepackage{hyperref}
\usepackage{mathrsfs}
  \usepackage{stackengine}
 
\newcommand{\cA}{\mathcal{A}}
\newcommand{\cD}{\mathcal{D}}
\newcommand{\cF}{\mathcal{F}}

\newcommand{\cX}{\mathcal{X}}

\newcommand{\cp}{\check{p}}
\newcommand{\sW}{\mathscr{W}}
\newcommand{\sS}{\mathscr{S}}
\newcommand{\sB}{\mathscr{B}}

\newcommand{\fK}{\mathfrak{K}}

\def\e{\varepsilon}
\newcommand{\A}{\mathbb{A}}

\newcommand{\R}{\mathbb{R}}

\def\be{\mathbf{e}}
\def\bp{\mathbf{p}}
\def\bq{\mathbf{q}}

\def\fC{\mathfrak{C}}

\newtheorem{thm}{Theorem}
\newtheorem{lem}{Lemma}
\newtheorem{cor}{Corollary}
\newtheorem{prop}{Proposition}

\newtheorem{dfn}{Definition}

\newtheorem{rem}{Remark}

\title[Learning on hexagonal structures and Monge--Ampère operators ]{Learning on hexagonal structures and Monge--Ampère operators}
\author{Noemie C. Combe}

\begin{document}
\maketitle

\begin{abstract}
Dually flat statistical manifolds provide a rich toolbox for investigations around the learning process. We prove that such manifolds are Monge-Ampère manifolds. Examples of such manifolds include the space of exponential probability distributions on finite sets and the Boltzmann manifolds. Our investigations of Boltzmann manifolds lead us to prove that Monge-Ampère operators control learning methods for Boltzmann machines. Using local trivial fibrations (webs) we demonstrate that on such manifolds the webs are parallelizable and can be constructed using a generalisation of Ceva's theorem. Assuming that our domain satisfies certain axioms of 2D topological quantum field theory we show that locally the learning can be defined on hexagonal structures. This brings a new geometric perspective for defining the optimal learning process.

\end{abstract}

\, 


\, 


\section{Introduction}

A rich differential geometric approach can be used in information geometry to study topics important to statistics \cite{Am,Boy1,Boy2,CM,CMM2022,CMM2023}, such as learning problems \cite{BCN,CN,AHTlearning,Hinton1,Hinton2}. This rich subject includes neural networks, decision theory~\cite{Ch}, manifolds of probability distributions (on discrete or continuous sample spaces) and learning methods such as the Ackley--Hinton--Sejnowski method \cite{AHTlearning}.

\,

Consider a Boltzmann machine, for instance. It is a network of stochastic $N$ neurons. In our setting, we define a neural network as a graph $G=(V,E)$ with a set of vertices $V$ being the neurons and the set of edges $E$, allowing the synaptic connections between the neurons. We assume that the number of vertices/edges is finite and take $|V|=N$.

\,

To each neuron one associates a finite set of values. In the simplest cases, neurons can take for instance values in the set $\mathbb{F}_2=\{1,0\}$, which corresponds to whether a given neuron is excited or not, respectively. The vector of values attached to the set of vertices ${\bf x}=(x_1,x_2,\cdots,x_N)$ is the {\it state} of the Boltzmann machine. 

\, 

Given two vertices $i$ and $j$ connected by an edge $e_{ij}$ there corresponds a real (positive) value $w_{ij}\in \R^+$ which is the weight of the edge. We can encode the collection of weights in an adjacency matrix $(w_{ij})_{1\leq i,j\leq N}$. This is a symmetric matrix, where entries are the weights $w_{ij}$. Given that self-loops are not allowed in the graph (i.e. there do not exist edges starting and ending on the same vertex) we have automatically that $w_{ii}=0$, implying that the symmetric weighted matrices are of null trace. 

\, 

The set of all such Boltzmann machines with fixed topology forms a geometric manifold, where (modifiable) synaptic weights of connections specify networks. Define the  space of weight parameters (this is a subspace of an Euclidean space $\mathbb{R}^{M}$ of symmetric matrices with trace equal to 0). 

\, 

let $\sB$ be the set of all probability distributions (given by Eq.\ref{E:logP}) provided by the Boltzmann machines. The connection matrix $W$ specifies each distribution uniquely. A Boltzmann manifold is therefore defined by $\sB$ and its coordinates are given by the $w_{ij}$.

\, 

By \cite{shoda}, any weight matrix $W$ of trace 0 can be expressed as a commutator i.e. $W=XY-YX=[X,Y]$, where $X,Y$ are matrices in $\mathbb{R}^{N(N+1)/2}$. Therefore, we can say that the space which parametrizes the manifold of probability distributions is a space of commutators. Note that the abelian case is, in some sense, pointless to consider since it corresponds to the case where there are no interactions occurring between the neurons.    

\, 

This commutator space parametrises a statistical manifold $M$ of probability distributions of exponential type, defined over the state space and consisting of $2^N$ states ${\bf x}$. Each  point on this manifold corresponds to a probability distribution. 

\, 

The procedure of learning corresponds ---on both manifolds---to a curve. We shall consider the learning process, locally. In particular, we show that the learning curve can be done in an optimal way, by using optimal transport provided by a Monge--Ampere operator. This is possible due to the fact that on the Boltzmann neural network manifold there exists a natural invariant Riemannian metric and a dual pair of affine connections. By ~\cite{CM,C24}  this structure is related to a Monge--Ampere manifold. Therefore, in Theorem~\ref{T:MA} we  show that:

\, 

\begin{center}
\begin{minipage}{11cm}
Theorem \ref{T:MA} and Corollary \ref{C:optimal} : {\it Everywhere locally on a dually flat manifold $M$ there exists an optimal learning procedure from one point to another, controlled by a pair of Monge--Ampere operators.}
\end{minipage}
\end{center}

\, 

Note that this statement is general (not only for discrete sample spaces). 

Having introduced the notion of ``very optimal learning'' in Definition \ref{D:veryoptimal}, we provide  the following statement:

\, 

\begin{center}
\begin{minipage}{11cm}
Theorem~\ref{P:Hex}: {\it
Let $\cD\subset N$ be a domain lying in a totally geodesic submanifold $N$ of a dually flat manifold $\sS$. Suppose that $\cD$ satisfies additionally Eq.~\ref{E:wdvv}. Then, $\cD$ is equipped with a hexagonal web. The learning on $\cD$ can thus be described using a local hexagonal lattice and it is very optimal.}
\end{minipage}
\end{center}
\, 

A flat totally geodesic satisfying the axioms of a Frobenius manifold has interesting geometric features. Frobenius manifolds are affine/Hessian manifolds such that the tangent sheaf carries the structure of a Frobenius algebra. A Frobenius algebra is a commutative, associative, unital algebra where the multiplication operation is invariant under the symmetric bilinear product.

Rewriting the associativity of the multiplication operation $\circ$ in the usual way as $( \partial_a\circ \partial_b)\circ \partial_c
= \partial_a\circ (\partial_b\circ \partial_c )$ we obtain a non--linear system of {\it Associativity Equations}, partial
differential equations for $\Phi$:
 \begin{equation}\label{E:wdvv}
\forall a,b,c,d :\quad \sum_{ef} \Phi_{abe}g^{ef}\Phi_{fcd} =  \sum_{ef} \Phi_{bce}g^{ef}\Phi_{fad},     
 \end{equation}

where $\Phi$ is a potential function and $\partial_a=\frac{\partial}{\partial_a x_a} $are flat vector fields.
In the community of physicists, they
are known as {\it WDVV} (Witten--Dijkgraaf--Verlinde--Verlinde) equations.

\, 

We prove  the existence of differential geometric webs of hexagonal type on these manifolds (see Thm. \ref{P:Hex}). Choosing the webs so that they form regular hexagons, this can enable us to provide our domain $\cD$ (everywhere locally) with a regular hexagonal/ honeycomb lattice. 

Hexagonal lattices simplify considerations around our {\it statistical data} greatly. In particular, if the hexagons are regular, one can use the (regular) lattice to consider the learning process locally  and benefit from dihedral symmetries arising from the hexagons and group web properties. 

\,

\thanks{
{\bf Acknowledgements} This research is part of the project No. 2022/47/P/ST1/01177 co-founded by the National Science Centre  and the European Union's Horizon 2020 research and innovation program, under the Marie Sklodowska Curie grant agreement No. 945339 \includegraphics[width=1cm, height=0.5cm]{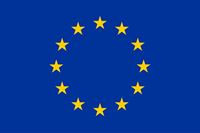}.}

\section{Networks and manifolds of probability distributions}
\subsection{Neural networks}
As stated in the introduction, we 
take as a starting point the example of some neural network coming from the Ising--Bolzman framework. Those networks are connected graphs without self-loops i.e. we omit the cases where edges are connected to only one vertex. 

\, 

The initial motivation for the study of manifolds of probability distributions is that one obtains a richer amount of information on the behavior of such neural networks if we study them not as individual objects but rather as a {\it family} of networks with fixed topology i.e. satisfying a given set of properties. This set of properties includes the number of vertices/edges of a graph as well as the set of decorations. By decorations we refer to threshold values (set of values attributed to each vertex) and weights of the edges.

\,

A neural network here is a decorated graph $G=(V,E,W,H)$, where $V$ refers to a set of $N$ vertices (the neurons) and $E$ refers to the set of edges. The symbol $W$ refers to the  $N\times N$-weight matrix, where each entry $w_{ij}$ gives us the weight of the edge connecting the vertex $i$ and the vertex $j$. Finally, the set $H$ corresponds to the set of values that each vertex can take. 

\, 

Given a family of neural networks with specified topology the collection of such neural networks forms a manifold. 

\, 

Each network corresponds to a point on that manifold. The intrinsic geometry becomes a tool for studying the learning process of the networks. This includes considering the intrinsic distance between two distinct points  as well as understanding the local curvature of the manifold, in a neighborhood of those points (that means networks), and how it impacts the learning process. 

\, 

In particular, we highlight that there exist two notions associated with this family of networks and which are tightly related. The first one is the space of weight matrices $\sW$: this forms a space of all symmetric matrices with null trace.  The second one is the corresponding statistical manifold $\sS$ which is defined over the space of states. This is given by a family of parametrised probability distributions $\sS=\{\rho(x;\theta)\}$ of exponential type. Any point on that manifold refers to an exponential probability distribution.  

\, 
\subsection{Statistical models}
Let $(\Omega,\mathcal{F},P)$ be a  measure space. A measure $P$ is said to be absolutely continuous w.r.t the measure $\mu$ if for every measurable set $C$, the equality $\mu(C)=0$ implies $P(C)=0.$  The measurable function is given by $P(C)=\int_C  \rho d\mu,$ $\forall C \subset \mathcal{F}$, where $\rho$ is called the density of the measure $P$ and $\rho= \frac{dP}{d\mu}$ is the Radon--Nikodym derivative.

\,

A statistical model (or $n$-dimensional manifold) can be defined as a parametrized family of probabilities, denoted $P_\theta$, being absolutely continuous with respect to a $\sigma$-finite measure $\mu$. Equivalently, one can consider it as the parametrized family of probability distributions $\sS=\{p({\bf x};\theta)\}$, where $p({\bf x};\theta) =\frac{dP_\theta}{d\mu}$ is the Radon--Nikodym derivative of $P_\theta$ w.r.t. $\mu$. 

\, 

This space comes equipped with:
\begin{itemize}
\item the {\it canonical parameter} given by $\theta= (\theta^1,\dots, \theta^n)\in \R^n$;
\item a family of random variables ${\bf x}=(x_i)_{i\in I}$  on a sample space $\Omega$, the letter $I$ refers to a set;  
\item $p({\bf x};\theta)$, the density of probability of the random variable ${\bf x}$ parametrized by $\theta$ w.r.t some measure $\mu$ on $\Omega$.   
\end{itemize}
  
A family $\sS=\{p({\bf x};\theta)\}$ of distributions is {\it exponential} if one can always write the density functions in the following way:   
  \[p({\bf x};\theta)= \exp(\sum_{i=1}^n\theta^ix_i-\Psi(\theta)),\] where   
  \begin{itemize}
  \item the symbol $\Psi(\theta)$ stands for a potential function. It is given in general by \begin{equation}\label{E:psi}
      \Psi(\theta)=\log\int_{\Omega}\exp\{\theta^ix_i\}d\mu\quad
  \end{equation}
  \item the vector parameter $\theta$ and co-vector of directional sufficient statistics  ${\bf x}=(x_i)_{i\in I}$ ($I$ is a finite set) have been chosen adequately;
\item the  canonical parameter satisfies $\Theta:=\{\theta\in \R^n: \Psi(\theta)<\infty\}$.
\end{itemize}
Note that in the discrete setting the integral is replaced by a sum in equation \ref{E:psi}.

\,

More precisely, the canonical parameter/coordinate system is obtained by expanding the $\log$ of the probability distribution:
\[\log p({\bf x})=\sum\theta^{i_1}_{1}x_{i_1}+\sum\theta^{i_1i_2}_{2}x_{i_1}x_{i_2}+\sum\theta^{i_1i_2i_3}_{3}x_{i_1}x_{i_2}x_{i_3}+\cdots \]\[+\sum\theta^{i_1i_2\cdots i_n}_{n}x_{i_1}x_{i_2}\cdots x_{i_n}
-\psi(\theta)\] where we assume $i_1<i_2<\cdots i_n$ and $(\theta^{i}_{1},\cdots, \theta^{i_1i_2\cdots i_n}_{n})$  form polyvectors. Therefore, we have $2^N-1$ parameters defining the coordinate system: $(\theta^{i}_{1}\cdots, \theta^{i_1i_2\cdots i_n}_{n})$ and corresponding to the probability distribution  $p({\bf x})$.

It is natural to include in the statistical model a structure of $n$-dimensional manifold, whenever $p({\bf x};\theta)$ is smooth enough in $\theta$.

\subsection{}
As a matter of clarity, let us return to the fundamental definitions of those objects. We refer to \cite{Lehman} for details. This definition applies to continuous and discrete sample spaces.  Let $\sS = \{ P_{\theta}, \theta\in \Theta\}$ be a family of probability distributions defined over a common sample space $( \Omega, \cF )$. 

\,

Assume ${\bf x}(\omega)$ is a statistic, $\omega\in \Omega$. We say that ${\bf x}$ is a sufficient  statistic for  $\sS$ if for  every $A \in \cF$  there exists a determination of the conditional probability function $P_{\theta}(A \vert t)$ being independent from $\theta$.

\begin{thm}[p.49\cite{Lehman}]
%$\cX$ is an be the space of measurable functions on $(\Omega, \cF )$.$(\cX,\cA)$ is a measurable space

If $\cX$ is Euclidean, and if the statistic ${\bf x}$ is sufficient for $\sS$, then there exist determinations of the conditional probability distributions $P_{\theta}(A \vert t)$ which are independent from $\theta$ and such that for each fixed $t$, $P(A \vert  t)$ is a probability measure over $\cA$.
\end{thm}
We invoke a second important statement. 
\begin{thm}[Factorization Theorem]
If the distributions $P_{\theta} \in \sS$  have probability distribution densities  $p_{\theta}= \frac{dP_{\theta}}{d\mu}$ with respect to a $\sigma$-finite measure $\mu$, then ${\bf x}$ is sufficient for $\sS$ if and only if there exist nonnegative Borel-measurable functions  $g_{\theta}$ on ${\bf x}$ and a nonnegative $\cF$-measurable function $h$ on $\cX$ such that

\[p_{\theta}(\omega)= g_{\theta} [({\bf x}(\omega)]h(\omega)
.\]

\end{thm}

\begin{proof}
In  \cite{Lehman} p. 49--50. 
\end{proof}

\subsubsection{The exponential family}
The exponential family can be defined in the following way:
\begin{dfn}

A set of probability distributions $ P[\cdot]$ is called an exponential family of finite dimension with canonical affine parameter if it is described by a family of densities):
\begin{equation}\label{E:DM45}
\frac{d P_\theta}{d \mu} (\omega)=p ( \omega; \theta) = p_0 (\omega) \exp \{ \theta^j x_j (\omega) - \Psi (\theta) \} ,
\end{equation}
where vectors $\theta=(\theta^1,...,\theta^n)$  are the parameters. The set $\Theta$ of parameter points $\theta=(\theta^1,\cdots,\theta^n)$ is the natural parameter space of the exponential family, 
co-vectors ${\bf x}= (x_1(\omega),...,x_n(\omega))$ are directional statistics for a given family of probability distribution, $\mu[\cdot]$ is a fixed dominating measure, over a sample space $(\Omega, \cF)$ and $\Psi (\theta) $ is normalising divisor, such that
\begin{equation}\label{E:DM45}
\exp\{\Psi (\theta)\}=\int_{\Omega} \exp \{\theta^j x_j (\omega)\}p_0 (\omega) \mu[d\omega].
\end{equation} 

It is assumed that the distribution parameter $\theta$ runs through all values for which the normalising divisor is finite. The directional statistic ${\bf x} = (x_1(\omega),...,x_n(\omega))$ is a co-vector of an exponential family.
\end{dfn}

Note the following fact:
\begin{lem}
The natural parameter space of an exponential family is convex.
\end{lem}

\begin{proof}
Let $\theta=(\theta^1,....,\theta^k)$ and $\bar\theta=(\bar\theta^1,....,\bar\theta^k)$ be two vectors from
 parameter space $\Theta$. Let us assume that the normalising divisors $ \Psi (\theta)$ and $\Psi (\bar\theta)$ are finite. So, by H\"{o}lder's inequality: 

\[ \begin{aligned}
\int \exp \left\{ \sum_i[ \alpha \theta^i + (1- \alpha) \bar\theta^i]x_i (\omega) \right\} d\mu(\omega )\\
\leq \left[ \int \exp
\left\{  \sum_i \theta^i x_i (\omega) \right\}\right]^{\alpha}&  \left[ \int \exp \left \{\sum_i \bar\theta^i x_i (\omega)\right\} d\mu(\omega) \right]^{1-\alpha}.
\end{aligned}\]
Assume that the convex set $\Theta$ lies in a $k$-dimensional linear space. So, 

$$p_{\theta}(\omega) = C(\theta) \exp [ \sum_{j=1}^k \theta^j x_j(\omega)],$$

where $C(\theta)$ is a constant, can be rewritten  in a form involving fewer than $k$ components of ${\bf x}$. Hence, we can assume that $\Theta$ is $k$-dimensional. 

From the factorisation theorem it follows that ${\bf x}(\omega) = (x_1(\omega),\cdots,
x_k(\omega))$ is sufficient for $\sS = \{ P_{\theta}, \theta \in \Theta\}$.

\end{proof}

Let us consider the following Lemma.

\begin{lem}
Let $\Theta\in \mathbb{C}^n$ be a set of values of the $n$-dimensional (complex) vector $\theta=(\theta^1,...,\theta^n)$ and such that the 
following integral 
\[
 \Phi(\theta) = \int_{\Omega} \phi (\omega) \exp \{\theta^j x_i (\omega) \} P\{d\omega\}
\]
is absolutely summable to a finite value. Then, $\Theta$ is a convex subset of an $n$-dimensional space.
\end{lem}

The integral $\Phi(\theta)$ is an analytic function of the complex parameters $(\theta^1,...,\theta^n)$, whenever their real parts vary in an open convex set of interior points of $\Theta$ i.e in $$\{ \theta: Re(\theta) \in Int( \Theta)\}.$$ The integral $\Phi$  is a continuous function of the parameter on any $m$-dimensional simplex, $m<n$, lying entirely in $\Theta$.

\,

The derivatives of $\Phi(\theta)$ with respect to the components $(\theta^1,...,\theta^n)$ of the vector $\theta$ in the open domain $Int (\Theta)$ may be evaluated by successive differentiation under the integral sign:

\begin{equation}
\frac{d^k}{ds^k} [ \phi(\omega)\exp\{\theta^j x_j(\omega)\}] = \phi (\omega) \left[ \prod_{j=1}^n
[x_j (\omega)]^{k_j}\right] \exp \{\theta^j x_j (\omega)\}
\end{equation}
where $k=(k^1,...,k^n)$.

\, 

Suppose the cube $C= \{  \theta: \vert \theta^j - \theta^j_0 \vert \leq r, \quad  j= 1,...,n \}$ 
is a subset of $\Theta$. Then in any concentric subcube $C_k = \{  \theta: \vert \theta^j - \theta^j_0 \vert \leq (1- \kappa)r, \quad  j=1,...,n,  \}$ 
$0< \kappa <1$, the derivative is majored by

\begin{equation} \label{E:maj}
\left[  \prod_{j=1}^ n k_j ! (\kappa r )^{-k_j} \right ] \sum_{\beta=1}^{2^n} \vert \phi(\omega) \vert \exp \{  \theta^{\beta }\cdot x(\omega)\},
\end{equation}
where $\theta^{\beta} \in \Theta$, $ \beta = 1,..., 2^n$, are the vertices of the cube $C$.

\subsection{Dually flat Riemannian manifolds}\label{S:Duallyflat}
We focus on the case where the sample space is discrete. Let $\sS=\{p({\bf x})\}$ be a set of probability distributions over the state space $\Omega$ consisting of $2^N$ Boltzmann states. 

\, 

The manifold $\sS$ is a $2^N-1$-dimensional manifold. This is easy to see, since $\sum\limits_{\bf x}p({\bf x})=1$ and it implies that there are $2^N-1$ degrees of freedom. For technical reasons, we assume that $0<p({\bf x})<1$, which implies that we are considering an open simplex of $2^N-1$ dimensions.

\, 

This manifold is equipped with a Riemannian metric (Fisher metric in statistics;  Hessian metric in analysis) and has a dual pair of affine (dual) connections \cite{Am,Ch}. It has been proved that the set of all probability distributions over a finite set forms a {\it dually flat manifold} \cite{Am}. We elaborate on this in the following paragraphs. However, let us first explain what such a structure implies algebraically. 

\, 

By Amari--Nagaoka, there exists a certain duality theory of the $\pm\alpha$-connections which comes into light as we consider the set of all probability distributions over a finite set. Works in \cite{CM,CCN-Algebra} suggest that this duality theory follows in fact from the existence of a hidden paracomplex structure on the manifold. Following this viewpoint, one considers the manifold endowed with a dual pair of affine connections rather as a certain manifold defined over an {\it algebra of paracomplex numbers}. 

\, 

 The algebra of paracomplex numbers is defined as the real vector space $\fC = \R\oplus \R$ with the multiplication
\[
(x,y) \cdot (x',y') = (xx' + yy', xy' + yx'). 
\]
Put $\varepsilon : =(0,1)$.  Then $\varepsilon^2 =1$, and moreover
\[
\fC =\R+\e\R= \{z=x+\e y \, |\, x,y \in \R \}.
\]

Given a paracomplex number $z_{+} = x+\varepsilon y$, its conjugate is defined by $z_{-}:= x-\e y$.

Paracomplex numbers also go by the name of split complex numbers and hyperbolic numbers. The split complex numbers terminology is rather used on a  number theory side, as it comes from a {\it split algebra} while the hyperbolic number terminology comes from the fact that it has been extensively used in  hyperbolic geometry. We refer for instance to~\cite{Albert}, where this notion has been deeply investigated. One can mention also the article~\cite{Atiyah} for a geometric use of splitcomplex/paracomplex numbers.  

\,

%The algebra of paracomplex number is of type $\langle\ 1,\, \epsilon \, |\, \epsilon^2=1\rangle$. 

Paracomplex numbers form a two dimensional algebra which coincides with the Clifford algebra of type $\mathcal{C}_{1,0}$. Under an adequate change of variables, one can also express the generators by a pair of two idempotent elements.

\, 

Let $E_{2m}$ be a $2m$-dimensional real affine space. {\it A paracomplex structure} on $E_{2m}$ is an endomorphism $\fK: E_{2m} \to E_{2m}$ such that $\fK^2=I$,
and the eigenspaces $E_{2m}^+, E_{2m}^-$ of $\fK$ with eigenvalues $1,-1$ respectively, have the same dimension. 
The pair $(E_{2m},\fK)$ will be called a {\it paracomplex affine space.}

\, 

Finally, {\it a paracomplex manifold} is a real manifold $M$ endowed with a paracomplex structure $\fK$ that admits an atlas of paraholomorphic coordinates (which are functions with values in the algebra $\fC = \R + \e\R$ defined above), such that the transition functions are paraholomorphic.
\, 

Explicitly, this means the existence of local coordinates $(z_+^\alpha, z_-^\alpha),\, \alpha = 1\dots, m$ such that
paracomplex decomposition of the local tangent fields is of the form
\[
T^{+}M=span \left\{ \frac{\partial}{\partial z_{+}^{\alpha}},\, \alpha =1,...,m\right\} ,
\]
\[
T^{-}M=span \left\{\frac{\partial}{\partial z_{-}^{\alpha}}\, ,\, \alpha =1,...,m\right\} .
\]
Such coordinates are called {\it adapted coordinates} for the paracomplex structure $\fK$.
\vspace{3pt}

If $E_{2m}$ is already endowed with a paracomplex structure $\fK$ as above,
we define {\it the paracomplexification of $E_{2m}$} as $E_{2m}^\fC = E_{2m} \otimes_{\R} \fC$ and we extend $\fK$ to a $\fC$-linear endomorphism $\fK$ of $E_{2m}^\fC$. Then, by setting
\[
E_{2m}^{1,0} = \{v\in V^\fC \, |\, \fK v=\e v\}=\{v+\e\fK v\, |\, v \in E_{2m}\},
\]
\[
E_{2m}^{0,1} = \{v\in V^\fC \, |\, \fK v= -\e v\}=\{v-\e\fK v\, |\, v\in E_{2m}\},
\]
we obtain $E_{2m}^\fC =E_{2m}^{1,0}\oplus E_{2m}^{0,1}$.

We associate with any adapted coordinate system $(z_{+}^{\alpha}, z_{-}^{\alpha})$ a paraholomorphic coordinate system $z^{\alpha}$ by 
\[
z^\alpha\, =\, \frac{z_{+}^{\alpha}+z_{-}^{\alpha}}{2} +\e\frac{z_{+}^{\alpha}-z_{-}^{\alpha}}{2}, \alpha=1,...,m .
\]

\,
 
We define the paracomplex tangent bundle as the $\R$-tensor product $T^\fC M = TM \otimes \fC$ and we extend the endomorphism $\fK$ to a $\fC$-linear endomorphism of $T^\fC M$. For any $\cp \in M$, we have the following decomposition of $T_{\cp}^\fC M$:
\[
T_{\cp}^\fC M=T_{\cp}^{1,0}M \oplus T_{\cp}^{0,1}M\,
\]
where 
\[
T_{\cp}^{1,0}M = \{v\in T_{\cp}^\fC M | \fK v=\e v\}=\{v+\e \fK v| v \in E_{2m}\} ,
\] 
\[
T_{\cp}^{0,1}M = \{v\in T_{\cp}^\fC M | \fK v= -\e v\}=\{v-\e \fK v|v\in E_{2m}\}
\]
are the eigenspaces of $\mathfrak{K}$ with eigenvalues $\pm \e$.
The following paracomplex vectors 
\[
\frac{\partial}{\partial z_{+}^{\alpha}}=\frac{1}{2}\left(\frac{\partial}{\partial x^{\alpha}} + \e\frac{\partial}{\partial y^{\alpha}}\right),\quad \frac{\partial}{\partial{z}_{-}^{\alpha}}=\frac{1}{2}\left(\frac{\partial}{\partial x^{\alpha}} - \e\frac{\partial}{\partial y^{\alpha}}\right)
\]
form a basis of the spaces $T_{\cp}^{1,0}M$ and $T_{\cp}^{0,1}M$.

\, 
 
 In particular, this allows us to define a paracomplex Dolbeault complex.
See \cite{CM} for more details and the bibliography therein. 
\, 

This approach is in fact quite natural, given that a paracomplex manifold is automatically endowed with a pair of real affine connections satisfying a duality relation. In particular, the $\R$-realization of such a manifold enables us to have a real manifold equipped with an involution. Going back to the manifold of probability distributions, the duality relation between the pair of connections results therefore from the existence of this involution on the manifold. 

\,

A manifold is said to be $\nabla$-flat (i.e. flat with respect to the covariant derivative) it is also $\nabla^*$-flat, where $\nabla^*$ is dual to $\nabla$. There exist two flat coordinates for dually flat manifolds. In such dually flat manifolds, there exist two specific coordinate systems: a $\nabla$-affine coordinate system $\eta$ and $\nabla^*$-affine coordinate system $\eta^*$.  The coordinate systems $\eta$ and $\eta^*$ are mutually dual if their natural bases are biorthogonal.

\, 

Note that in a Riemannian manifold, dual coordinate systems may not exist. Nevertheless, they always exist on an $\alpha$-flat manifold. One consequence of this structure follows from \cite[Thm 3.4]{Am}.

\, 

Given a manifold, it is endowed with a pair of dual coordinates systems if and only if there exists a pair of corresponding potential functions such that the metric tensor is derived by differentiating it twice i.e. \[g_{ij}=\partial_i\partial_j\psi(\eta),\quad g^{ij}=\partial^i\partial^j\psi^*(\eta^*),\]
where $\partial_i=\frac{\partial}{\partial \eta^i}$ and $\partial^i=\frac{\partial}{\partial \eta^*_i}$.

\,

The passage from one dual coordinate to the other is done using Legendre transforms i.e. $$\eta^i =\partial^i\psi^*(\eta^*)\quad \text{and}\quad \eta^*_i =\partial_i\psi(\eta).$$

Whenever, the manifold is flat with respect to a pair of torsion-free dual affine connections $\nabla$ and $\nabla^*$, there exists a pair of dual coordinate systems $(\eta,\eta^*)$ such that $\eta$ is a $\nabla$-affine and $\eta^*$ is a $\nabla^*$-affine coordinate system see \cite[Theorem 3.5]{Am}.

\subsection{The Monge-Ampèreness of dually flat manifolds}
The previous section~\ref{S:Duallyflat} recalled the notion of dually flat connections and their impact on the geometry of the manifold. In particular, this has strong implications. One of them is the direct relation to Monge--Ampère operators. 

\, 

Originally, the Monge problem was to find the optimal transport $T:\R^d\to \R^d$ between two distribution of masses $\rho_1$ in a domain $A\subset \R^d$ to a distribution $\rho_2$
on a domain $B\subset \R^d$ such that $$\int_A \rho_1 dx=\int_B \rho_2 dx$$ and the total mass remained preserved under transport. 

\,

If $\cD$ is a strictly convex bounded subset of $\mathbb{R}^n$ then for any nonnegative function $f$ on $\cD$ and continuous $\tilde{g}:\partial \cD \to \mathbb{R}^n$ there is a unique convex smooth function $\Phi\in C^{\infty}(\cD)$ such that 
\begin{equation}\label{E:EMA}
\det \mathrm{Hess}(\Phi)= f, 
\end{equation} in $\cD$ and $\Phi=\tilde{g}$ on $\partial \cD$.

\,

The {\it geometrization of an elliptic Monge--Ampère equation} refers to the geometric data generated by $(\mathscr{D}, \Phi)$, where \begin{itemize}
    \item $\mathscr{D}$ is a strictly convex domain 
    \item $\Phi$ a real convex smooth function (with arbitrary and smooth boundary values of $\Phi$)  
    \end{itemize}
    such that Eq.~\eqref{E:EMA} is satisfied.   
 
\,

\begin{dfn}
    A Monge--Ampère manifold is manifold on which everywhere locally there a smooth potential function $\Phi$ satisfying an equation of Monge--Ampère type.
\end{dfn}

We prove the following geometric statement. 

\,

\begin{thm}\label{T:MA}
Consider a dually flat manifold. Then, this manifold forms a pair of Monge--Ampere manifolds defined in the $\eta$ coordinates (respectively in the $\eta^*$ coordinates). 
\end{thm}
\begin{proof}

As a result of \cite[Theorem 3.5]{Am}: if the manifold is flat with respect to a pair of torsion-free dual affine connections $\nabla$ and $\nabla^*$ then there exists a pair of dual coordinate systems $(\eta,\eta^*)$ such that $\eta$ is a $\nabla$-affine and $\eta^*$ is a $\nabla^*$-affine coordinate system.
   
    \,
In particular, in those coordinate systems everywhere locally the metric tensor is derived by differentiating twice a potential function $\psi(\eta)$ (resp. $\psi^*(\eta^*))$. The metric is non-degenerate which means that $\det(\mathrm{Hess}(\psi))\neq 0$.

 Therefore, we have $\det(\mathrm{Hess}(\psi))=f$ where $f$ is a real function or a constant. There is no need to discuss the Dirichlet conditions for the boundary since we have omitted the boundaries.  The same holds for the coordinates $\eta^*$.  Therefore, we can conclude that dually flat manifolds admit the structure of a Monge--Ampere manifold.
\end{proof}

\begin{cor}\label{C:optimal}
    Everywhere locally on a dually flat statistical manifold  there exists an optimal learning procedure from one point to another, controlled by a pair of Monge--Ampère operators.
\end{cor}
\begin{proof}
The proof follows from the above statement.
\end{proof}
\subsection{Statistical version of Monge-Ampèreness}\label{S:MAP}
\,

\subsubsection{}\label{S:8.3} We turn to a probabilistic version of the Monge-Ampère operator. 
Consider the (probability) measure space $(\R^d,\mathcal{F},\mu)$, where $\mathcal{F}$ is formed from Borel sets in $\R^d$ and $\mu$ is a Borel measure.
By $\mathcal{P}(\R^d)$ we denote the space of Borel probability measures on $\R^d$.

\, 

Let $T:\R^d \to \R^d$ be a Borel transformation defined $\mu$-almost everywhere such that if the measure of a measurable set $A\subset \mathcal{K}\subset \R^d$ (where $\mathcal{K}$ is a compact)  vanishes, then the measure of $T^{-1}(A)$ vanishes too. Given a Borel set $M \subset \R^d$ one can {\it pushforward} $\mu$ through $T$, resulting in a Borel probability measure $T_{\diamond}\mu$ on $\R^d$ given by: 
\begin{equation}\label{E:nu}
    T_{\, \diamond\, }\mu[M]:=\mu[T^{-1}(M)].
\end{equation}

By~\cite{Bre1}, there exists a unique map $h$ such that $$T=\nabla U\circ h,$$ where $h$ preserves the volumes and $U$ is a convex Lipschitz function in the neighborhood of $\mathcal{K}$. 

\, 

Obtaining the Monge-Ampère equation requires a few additional steps. Let $V^*$ be the Legendre transform of $U$ that is $V=U^*$. By Brenier's factorisation theorem~\cite{Bre1}, $\nabla U$ and $T$ map $\mathcal{K}$ into the same set $\mathcal{K}^*$ that is 
$T(\mathcal{K})=\nabla U(\mathcal{K})=\mathcal{K}^*$ and $\nabla V(\mathcal{K}^*)=\mathcal{K}.$
For any continuous function $f\in C^1$ on $\mathcal{K}$ we have the following:

\begin{equation}\label{E:MA-Brenier}
    \int_\mathcal{K} f(T(x))dx=\int_\mathcal{K} f(\nabla U(x))dx=\int_{\mathcal{K}^*} f(y)\det (D^2 V(y))dy
\end{equation}
where we have the change of variables $y=\nabla u(x)\in \mathcal{K}^*$ and $x=\nabla V(y)\in \mathcal{K}$.
One deduces the existence of a positive function $r(x)$, integrable in $\mathcal{K}^*$ such that 
\[\int_\mathcal{K} f(T(x))dx=\int_{\mathcal{K}^*} f(y)r(y)dy.\]
Therefore, we have that $V$ is a solution to the Monge--Ampère problem. The function $V$ is convex and we have $\det D^2 V(y)=r(y)$ almost everywhere on $\mathcal{K}^*.$

\,

A direct implication of the  statement in Theorem\ref{T:MA} is the corollary outlined below, for manifolds of probability distributions on finite sets.

\begin{cor}
   Consider a manifold of all probability distributions over a finite set. Then, this forms a Monge--Ampère  manifold. 
\end{cor}
\begin{proof}
 A manifold of all probability distributions over a finite set is a {\it dually flat manifold}. This follows from works by Amari \cite{Am}. As a result of applying Theorem \ref{T:MA} we conclude that a manifold of probability distributions over a finite set is a Monge--Ampère  manifold. 
\end{proof}

\section{Applications to Boltzmann manifolds}
In the following we discuss the case of Boltzmann manifolds, without hidden units. A Boltzmann machine without hidden units realizes a probability distribution $p({\bf x};\theta)$ depending on the weight matrix $W=(w_{ij})_{1\leq i,j\leq n}$ such that 
 \begin{equation}\label{E:logP}
     \log p({\bf x};\theta)=\sum\limits_{i>j}w_{ij}x_ix_j-\psi(w),
      \end{equation}
      where $\psi(w)$ is a potential function (also sometimes referred to as the normalizing constant).

\,

A probability distribution of the form Eq.~\ref{E:logP} is realized by a Boltzmann machine with connections $w_{ij}$ as its stationary distribution. $\sB$ is the set of all probability distributions provided by the Boltzmann machines. The connection matrix $W$ specifies each distribution uniquely. 
\begin{dfn}
    A Boltzmann manifold is defined by $\sB$ and its coordinates are given by the $w_{ij}$.
\end{dfn}

\, 

Looking more globally at the manifold of networks, it can happen that there exists a larger manifold into which this object embeds. It then becomes equally  interesting to study the relative geometry of this pair formed from the submanifold and the larger manifold. 

\, 

Typically here, the Boltzmann manifold inherits the geometric properties of $\sS$, since it is a submanifold of $\sS$. Therefore, it is dually flat and forms a Monge--Amp\`ere manifold, by Theorem~\ref{T:MA}. This is a very useful property  one can take advantage of for the learning process, since the Monge--Amp\`ere operator comes related to the notion of optimal transport.

\, 
\subsection{Weight matrices are commutators}
The following lemma enables us to define a parametrization morphism $\pi: \sB \to \sW$, where $\sW$ is the commutator of square matrices and $\sB$ is Boltzmann manifold i.e. an exponential family of probability distributions, being a submanifold of $\sS$.

\, 

\begin{prop}
Assume $\sB$ is a Boltzmann manifold where the Boltzmann machine is without hidden units. Then, the space $\sB$ is parametrised by a space of commutators of $N\times N$ matrices, where $N=|V|$ is the number of stochastic neurons.
\end{prop}
\begin{proof}
 Assume we have a family of networks of same topology and of Bolzman type forming a neural manifold.
\,

 Let $\sW$ be the space of weight matrices and let $\sS$ be the corresponding manifold of probability distributions.

 \, 

 A Boltzmann machine without hidden units realizes a probability distribution $p({\bf x};\theta)$ depending on the weight matrix $W=(w_{ij})$ such that equation \ref{E:logP} is satisfied

 Conversely, a probability distribution $p({\bf x};\theta)$ given by Eq.\ref{E:logP}
is realized by a Boltzmann machine with connections $w_{ij}$ as its stationary distribution.   

 \,
 Each distribution is specified uniquely by a weight matrix. On the manifold $\sS$ coordinates are given by the parameters $w$.

 \, 

 Therefore, there exists a (flat proper) morphism $\pi:\sS\to \sW$, where $\sW$ is the parameter space formed from symmetric matrices $N\times N$ with trace zero. 

 \, 

 By \cite{shoda}, any such matrix is a commutator i.e. for any $W$ such that $Tr(W)=0$ and $W^T=W$ there exists a pair of matrices $X,Y$ such that $[X,Y]=XY-YX=W$.

 \, 

 Therefore, the space $\sS$ is parametrised by a space of commutators of $N\times N$ matrices.
\end{proof}

\subsection{What we mean by learning}
We discuss how our previous results allow us to define some optimal learning trajectories, on the manifold. This is done via the Monge--Ampère operators. Note that following \cite{CM}, this can be related to Topological Quantum Field Theory, whenever a certain notion of flatness of the manifold occurs i.e. whenever the statistical manifold satisfies the associativity equations (i.e. Witten--Dijkgraaf--Verlinde--Verlinde partial differential equations).

\,

We apply this in particular to the Ising--Boltzmann manifold and use the Ackley {\it Et al.} learning methods.

%\begin{lem} PAS SURE A QUOI CA SERT ICI
%Assume $M$ is a statistical manifold on a finite set. 
%Then, everywhere locally on $M$ there exists a path which is controlled by a Monge--Ampere operator. 
%\end{lem}

\, 

Consider a Boltzmann manifold. For simplicity assume that there are no hidden elements. Let ${\bf x}$ be a given Boltzmann state and let $q({\bf x})$ be a probability distribution over ${\bf x}$.   The Boltzmann machine is required to modify its connection weights and thresholds so that the stationary distribution $p({\bf x})$ of the Boltzmann machine becomes as close to $q({\bf x})$ as possible.

Different methods exist to achieve this. For instance, the Ackley--Hinton--Sejnowski method provides a learning rule such that the average adjustment of the weights is given by the following formula:

\[\forall i,j\in \{1,\cdots,N\}, \quad  \Delta w_{ij}=\Delta w_{ji}=c\cdot (q_{ij}-p_{ij}),\quad \text{where:}\]
\begin{itemize}
    \item $c$ is a constant;
    \item $q_{ij}=\mathbb{E}_q[x_ix_j]=\sum\limits_{x} q({\bf x})x_ix_j$, is the expected value that both $x_i$ and $x_j$ are jointly excited under the probability distribution $q$;
    \item $p_{ij}=\mathbb{E}_p[x_ix_j]=\sum\limits_{x} p({\bf x})x_ix_j$, is the expected value that both $x_i$ and $x_j$ are jointly excited under the probability distribution $p$.
\end{itemize}
It was shown that this learning rule can be identified with 
\begin{equation}\label{E:Delta}
\Delta w_{ij}=-c\cdot\frac{\partial I(q,p)}{\partial w_{ij}},\end{equation} 
where $I(q,p)$ is the Kullback information. It is given by $I(q,p)=\sum\limits_{\bf x}q({\bf x})\log\frac{q({\bf x})}{p({\bf x})}$.

\begin{prop}
Suppose $\sB$ is a Boltzmann manifold. Then, for infinitesimally close Boltzmann machines with respective weight matrices $w_1$ and $w_2$ the learning rule defined in Eq.\ref{E:Delta} is controlled by a Monge--Ampère operator.     
\end{prop}

\begin{proof}
 L0. Assume we have Boltzmann manifold. 

\,
Suppose that $q$ is given by the weight $w_1$ and $p$ is given by the weight $w_2$ then the Kullback information is given by 

$$ I(p,q)=\sum p({\bf x},w_1)\log \frac{p({\bf x};w_1)}{p({\bf x};w_1)}.$$

\,
 
If we suppose that $p$ and $q$ are infinitesimally close which means that we have $w_1=w$ and $w_2=w+dw$ then the Kullback information is given by the (Hessian) metric $g$, that is: $I(q,p)=\frac{1}{2}\sum g_{ij,kl}dw^{ij}dw^{kl}.$ 

\,

L2. In other words, the metric tensor is given, in the flat coordinates, by the Hessian of a potential function: 

\[g_{ij,kl}=\frac{\partial^2}{\partial w^{ij}\partial w^{kl}} \phi(w).\]

\,

L3. Therefore, given that the metric is by definition non-degenerate (this forms from the construction of dually flat manifolds) the learning rule $\Delta w_{ij}$ defined in equation~\ref{E:Delta} is controlled by a Monge--Ampère operator. 

\,

To construct the potential function explicitly on the space of probability distributions, we can use Brenier's theorem and the construction outlined in Section\ref{S:MAP}.

\end{proof}

\section{Hidden local honeycomb lattices and symmetries related to Frobenius manifolds}
We propose a geometric tool in the scope of improving the current learning methods. This tools relies on the notion of webs. The webs on a domain $\cD$ in $\sS$ are hexagonal if $\cD$ meets the axioms of a Frobenius manifold \cite[p.19-21]{Man99}. A Frobenius manifold is a geometrization of the partially differential Witten--Dijkgraaf--Verlinde--Verlinde equation \ref{E:wdvv}.

One property of such manifolds is that the tangent sheaf $(T_\cD,\circ)$ has the structure of a Frobenius algebra (commutative, associative, unital with a certain invariance of the multiplication under a bilinear symmetric map: $\langle x\circ y,z \rangle= \langle x,y\circ z \rangle$ where $x,y,z \in T_{\cD}$) and that it is endowed with a flatness condition i.e. it is equipped with an affine space of flat connections given by: $$\nabla_{\lambda,X}(Y)= \nabla_{0,X}(Y)
+ \lambda (X\circ Y),$$ where $\lambda$ is a real parameter and $X,Y$ are vector fields.

\, 

In particular, an important result of \cite{C24}, stemming from  
(\cite[Thm 1.5]{Man99}), is that a pre-Frobenius manifold is a Frobenius manifold if and only if $\{\nabla_{\lambda}\}$ forms a pencil of flat connections. By~\cite{C24}, a pre-Frobenius manifold coincides with a Monge--Ampère manifold i.e. a manifold on which everywhere locally a Monge--Ampère equation is satisfied.  

\,

We introduce the notion of very optimal learning curves. Such curves result in a minimal loss of energy during the learning procedure due to the flatness of the space in which they are properly embedded. 

\, 

\begin{dfn}\label{D:veryoptimal}
   A learning curve is a curve on the statistical manifold illustrating a learning procedure. We call it very optimal if the curve lies on a flat manifold, where, everywhere locally, the Monge--Ampère equation is satisfied.    
   
\end{dfn} 

\subsection{Webs, foliations and local honeycomb lattices}
We consider webs in this section.
This implies introducing the tool of local trivial fibrations.

\subsubsection{}
Consider local trivial fibrations of class $C^k$. This is given by a triple $(Y,X,\pi)=\lambda$ where $\pi:Y\to X$  is a projection of $Y$ onto $X$ and $Y$ and $X$ are differentiable (smooth) manifolds of respective dimensions $m$ and $n$ where $m > n$. 
\begin{enumerate}
    \item For each point $x \in X$ the set $\pi^{-1}(x)\subset Y$ is a submanifold of dimension $m-n$ which is diffeomorphic to a manifold $F$;
 \item For each point $x \in X$ there exists a neighbourhood $U_x\subset X$ such that $\pi^{-1}(U_x)$ is diffeomorphic to the product $U_x \times F$, and the diffeomorphism between $\pi^{-1}(U_x)$ and  $U_x \times F$  is compatible with $\pi$ and the projection $pr_{U_x} : U_x \times F \to U_x$.
\end{enumerate}

Given our interest  in the local structure of such manifolds we can often think of them as connected domains of a Euclidean space of the same dimension.

If $T_{\cp}(Y)$ is the tangent space of $Y$ at a point ${\cp}$. A fibre $F$ passing through ${\cp}$ determines in $T_{\cp}(Y)$ a subspace $T_{\cp}(F_x)$, of codimension $r$ say, tangent to $F_x$ at ${\cp}$. We provide $T_{\cp}(Y)$ with a local moving frame $\{e_i,\,  e_\alpha; i=1,...,r; \alpha =r+1,...,m\}$, where $dim F =m - r$ and $dim X =r$. It is natural then to obtain a co-frame $\{\omega^i, \omega^\alpha\}$ dual to $\{e_i, e_\alpha\}$, such that  $\omega^i(e_{\alpha})=0$. Fibers of the fibration are integral manifolds of the system of equations $$\omega^i=0.$$

Since there exists a unique fibre $F$ through a point ${\cp}\in Y$, the system above $\{\omega^i=0, i=1,\cdots, r\}$ is completely integrable. By the Frobenius theorem, the integrability condition is given by 
\[d\omega^i=\sum_{j=1}^r\omega^j\wedge \phi^i_j\]
where $\phi^i_j$ are differential forms.

\, 

Assume $Y$ is an $m$-dimensional manifold. A $k$-dimensional distribution $\theta$ on an $m$-dimensional manifold $Y$, $0\leq k\leq m$ is a smooth field of $k$-dimensional tangential directions.  To each point ${\cp} \in Y$ there is a function which assigns a linear $k$-dimensional subspace of the tangent space $T_{\cp}(Y)$ to ${\cp}$. 

%A distribution is said to be integrable if through each point $p \in Y$ there passes a unique $k$-dimensional integral surface of 0 which is tangent to the distribution at any of its points. A distribution $\theta$ is integrable if and only if a system of differential equations defining 0 is completely integrable.

\,

%We are given a $k$-dimensional foliation on $M$ if the manifold $M$ is fibrated into $k$-dimensional surfaces. This means that through each point
%$p\in M$ there passes one and only one smooth $k$-dimensional surface which smoothly depends on $p$. 

\, 

These surfaces are called the leaves of the foliation. The numbers $k$ is the dimension of the foliation.

\, 

Let $Y=X$ be a differentiable manifold of dimension $nr$. We  say that a $d$-web $W(d,n,r)$ of codimension $r$ is given in an open domain $D\subset Y$  by a set of $d$ foliations of codimension $r$ which are in general position. 
\,

Consider  the following web $W(n+ 1, n, r)$ given in an open domain $D$ of a differentiable manifold of dimension $nr$. In a neighbourhood of
a point $p\in D$ a web $W(n +1, n, r)$ is defined by $n +1$ foliations $\lambda_1,\cdots, \lambda_{n+1}$ (in general position).
Each foliation can be defined by a completely integrable system of Pfaffian equations:
$$\omega_{\zeta}^j=0\quad \text{where} \quad \zeta=1,\cdots, n+1,\,  j=1,\cdots r.$$

Conditions of complete integrability of each of the systems must satisfy an equation of Maurer--Cartan type. Although there are $(n+1)r$ forms $\omega_{\zeta}^j$ on $X^{nr}$ only $nr$ of them are linearly independent. This gives an additional linear equation on  $\omega_{\zeta}^j$. 

\subsubsection{}
Web equivalence classes can be defined. Consider a web with $d$ foliations provided by 
\[\lambda_i(x)=k_i\quad 1\leq i\leq d,\]
where $\lambda_i$ are smooth functions, having non-null gradient and $k_i$ are constants. Then, the function $\lambda_i$ can be replaced by the function $g_i(\lambda_i)$ such that the gradient of $g_i$ is non-null, without changing the $i$-th foliation. This forms an equivalence class of $d$-webs on a given open domain of a manifold $M$.

\subsection{}

In \cite{CMM2022,CMM2023} the utility of web theory for models of databases subject to noise such as spaces of probability distributions on a finite set
turned out to be a successful tool for unraveling hidden symmetries of those spaces.

 This allowed us to demonstrate the existence of mathematical structures describing symmetries of relevant geometries where Commutative Moufang Loops (CML) play an important role.  Webs can be hexagonal, algebraizable, isoclinic, parallelizable, Grassmanniable. 

\begin{rem}
    Note that on a Frobenius (sub)manifolds, webs satisfy all those properties. 
\end{rem}
\,

 The left loops (loops with a left type of multiplication) are in bijection with left homogeneous spaces (which are equipped also with a left type of multiplication).  Therefore, we can state that left loops and  left homogeneous spaces describe the same structure on a given affine manifold with an affine connection.

\, 

Consider an affine  manifold endowed with its affine connection. The space of probability distributions on a finite set is an example of such an affine connection manifold~\cite{CM}.

\, 

 For those types of manifolds, one can introduce a specific type of local loop in the neighborhood of any point (see \cite{Kika} for details concerning the properties of the loop).

These loops are uniquely determined by means of the parallel transport along a geodesic, see \cite{Sa1,Sa2}. 

This family of local loops defines a covering of the manifold and uniquely determines all affine connections. 

Therefore, algebraic webs are an unavoidable structure controlling the geometrical structure of the manifold under investigation.

%\subsubsection{}
%Let $M$ be an $m$-dimensional manifold. Assume $f_1,\cdots, f_k\in C^{\infty}(M)$ are smooth functions on $M$. Their Jacobian matrices are of rank $k$ everywhere. The functions $f_1,\cdots, f_k$ are defined up to an arbitrary smooth transformation.

\,

%A $d$-web consists of $d$ foliations where the leaves are everywhere in general position. Let $\cD$ be an open domain of an $nr$-dimensional  differentiable manifold $M^{nr}$. The $d$-web is denoted by $W(d,n,r)$, where:
%\begin{itemize}
%\item $d$ is the number of foliations i.e. we have %\begin{equation}\label{E:fol}
  %  f_1(x)=\lambda_1,\cdots, f_d(x)=\lambda_d,
  %  \end{equation}
%%\item $r$ is the codimension of the foliations given by the equation Eq.\ref{E:fol};
%\item $m=nr$ is the dimension of the smooth manifold $M$.
%\end{itemize}
%An $(n-k)$-dimensional foliation (codimension $k$ foliation) has leaves defined by equations of the form 
%\[f_1(x)=\lambda_1,\cdots, f_k(x)=\lambda_k,\]
%where $\lambda_i$ are constants. 
%We call a web $k$-dimensional whenever its foliations are of dimension $k$.

\subsection{Hexagonal and parallelizable webs} 
The class of {\it hexagonal} webs is important. Let $\cD$ be, for simplicity, a two dimensional domain. Take ${\cp}\in \cD$ a point. Consider three (different) regular families of smooth curves on $\cD$ in general position. This constructs a web.

In the neighbourhood of the point ${\cp}$ this web allows the construction of a family of hexagonal figures. 
Depending on the geometry of the domain the hexagons might not be closed. If that is the case the web is not hexagonal. 

\,

Suppose we have families of 3-webs each  formed from a set of three parallel lines. The superimposition of those three families outlines closed hexagons. The case where the 3-webs are formed by three families of parallel straight lines is called {\it parallelizable}. If a 3-web is equivalent to a parallel web then it is called parallelizable. 

%\begin{exa} 
%One can consider in a $2r$-dimensional affine space $\A^{2r}$, three families of parallel $r$-dimensional planes, which are in general position. They form a three-web, which is called a {\it parallel} three-web. 
%\end{exa}

\,

Let us invoke the following statement:

\begin{thm}[Thm.1.5.2 and Cor. 1.5.4~\cite{Go88}]\label{T:Ash}
A web $W(n+1,n,r)$ for $n>2$ is parallelizable if and only if it is torsionless. %tensors $a^i_{\alpha \beta^{jk}}$ vanishes.
Moreover, given an $(n+1)$-web being parallelizable, all its $(k +1)$-subwebs are also parallelizable.

\end{thm}

We apply this statement directly to the manifolds of probability distributions. 

\begin{prop}\label{T:para}
Let $M$ be a statistical manifold satisfying the Frobenius manifold axioms. Suppose that it admits a web $W(n+1,n,r)$, with $n>2$. Then, the web $W(n+1,n,r)$ is parallelizable and its $(k+1)$-subwebs are parallelizable too.
\end{prop} 
\begin{proof}
By hypothesis if $M$ is a Frobenius manifold then the Chern connection on it vanishes and is torsionless. 
The rest of the proof follows from the application theorem~\ref{T:Ash}. Therefore, the web $W(n+1,n,r)$ is parallelizable, as well as all its  $(k +1)$-subwebs. 
\end{proof}
%%%%Theorem parallel

According to~\cite{Go88} a web $W(n+1,n,r)$ is said to be hexagonal if all $n+1\choose 3$  its 3-subwebs are hexagonal.

\, 

The vanishing of the symmetric part of the curvature tensor is a necessary and sufficient condition for hexagonality of a multicodimensional 3-web ~\cite[Sec 1.2]{Go88}.

\,

We prove the following statement. 
\begin{prop}
Consider a Frobenius manifold. Let $W(2,3,r)$ be a web, where $r\geq 1$. Then, this web  is hexagonal.  
\end{prop}

\begin{proof}
 Let $W(2,3,r)$ be a web on a Frobenius manifold. 
 Having the Frobenius condition implies that the Chern connection is flat and that the torsion vanishes. 
 By Prop. \cite{Ak} the web is hexagonal.
\end{proof}

%%%% 
\subsection{Parallelizable webs of the simplex}

Let $(\Omega,\cF,P)$ be a probability space, where the sample space is a finite discrete space $\Omega =\{\omega_0,x_1,\dots,\omega_d\}$  where the cardinality $\#\Omega =d+1$ and $\cF$ is the algebra of all measurable
subsets of $\Omega$, satisfying $\#\cF=2^{n+1}$. This algebra is generated by the family
$\{\{\omega_0\},\{\omega_1\},\dots \{\omega_d\}\}$  of the $d+1$ one point subsets of $\Omega$.

\,

The finite probability $P$, absolutely continuous with respect to the uniform measure, is well defined by its distribution function:
\begin{equation}
p^i=P[\{\omega_i\}], \qquad
 0< p^i<1,\qquad \sum\limits_{i=0}^d p^i=1.
\end{equation}

The distribution $p$ appears as a random variable $p:X \to
[0,1]$ which can be written as:
\begin{equation}
P=\sum_{i=0}^d p^i\delta_{\omega_i},\quad\text{ with } \quad0< p^i<1,\quad \sum\limits_{i=0}^d p^i=1.
 \end{equation}
where\[\delta_{\omega_k}(\omega)= \begin{cases}
1&\text{ if  $\omega=\omega_k$},\\ 0& \text {otherwise}.
\end{cases}
\]

We will construct some webs for the 2-dimensional case. This is naturally generalised to any dimension $n$.
Let $(\Omega_3, \A_3)$ be a measurable space where $\Omega_3=\{\omega_{1},\omega_{2},\omega_{3}\}$ and $\cF_3$  the algebra of the subset of $\Omega$, of cardinality $\# \cF_3=2^{3}$ (the algebra of events). 

Any probability distribution $P$ on the algebra $\A$ is given by a triplet $(p^{1},p^{2},p^{3})$ such that \[ p^{j}=P[\{\omega_{j}\}],\quad \text{where} \quad j=1, 2, 3.\]. 

We associate a probability distribution $P$ with a point of the simplex i.e.
\begin{equation}\label{E:11.1}
P\leftrightarrow \mathbf{p} \in \left\{(p^{1},p^2,p^3) \mid  \ p^{1}+p^2+p^3=1\ ;\ \forall \, j,\ p^{j}\geq 0 \right\}.
\end{equation}

In what follows we work on the simplex. Let $\vec e_1=(1,0,0),\, \vec e_2=(0,1,0),\, \vec e_3=(0,0,1)$ be the canonical basis of the affine space $\R^3$ where we fixed the origin at $(0,0,0)$. The probability distribution concentrated on the elementary events $\{\omega_{k}\}, \ P[\{\omega_k\}]=1$, can be identified with the basis vectors
\begin{equation}
\vec e_{k}=(e_{k}^{1},e_{k}^{2}, e_{k}^{3}), \text{ with } e_{k}^{j}=\delta_{k}^{j}.
\end{equation}

The family $\{\vec e_{k}\}_{k=1}^{3}$ defines a canonical basis of $\mathbb{R}^{3}$ ($\mathbf{e}_{i} \cdot \mathbf{e}_{j} = \delta_{i}^{j}$, where $\cdot$ denotes the scalar product in $\mathbb{R}^{3}$). The probability $P$ is thus in bijection with the following vector:
\begin{equation}\label{E:vectp}
\vec p= p^{1}\vec e_1+ p^2\vec e_2+p^3\vec e_3, \text{ where } p^{j}\geq 0,\ p^1+p^2+p^3=1.
\end{equation}

\begin{center}
\begin{tikzpicture}[scale=1.0]
 \draw(0.12,-0.2) node{$\mathbf{0}$};
 \draw [black,dashed,->] (0, 0) -- (0,2.9);
 \draw[black,dashed,->]  (0, 0) -- (3.0,-0.43);
\draw [black,dashed,->] (0, 0) -- (-1.98,-1.0);
\draw [black] (0, 3) -- (-2,-1);
\draw [black] (0, 3) -- (3.2,-0.45);
\draw [black] (-2, -1) -- (3.3,-0.45);
\draw(0.4,3.2) node{$\be_{3}=(0,0,1)$};
\draw(4.2,-0.5) node{$\be_{2}=(0,1,0)$};
\draw (-1.6,-1.2)node{$\be_{1}=(1,0,0)$};
\draw(0.2,1.5) node{$\vec{e}_{3}$};
\draw(1.2,0.1) node{$\vec{e}_{2}$};
\draw (-0.8,-0.1) node{$\vec{e}_{1}$};
\end{tikzpicture}
\end{center}

Any point $\bp$ of the simplex corresponds to a probability $P$ where $P[\{\omega_{j}\}]=p^{j}$ is in bijection with the  point $\vec p$ of the space $(\mathbb{R}^{3}, (0,0,0))$ and verifying~\eqref{E:vectp}.

\smallskip 

A point $\bp$ lies on the face of a simplex  with vertices $\mathbf{e}_{(j_{1})},\dots,\mathbf{e}_{(j_{k})}$, if and only if $P[\{\omega_{j(1)}\}]+\dots+P[\{\omega_{j(k)}\}]=1$. Thus, the face of the simplex forms the collection $\mathrm{Cap}(\Omega',\A')$ of probability distributions on the finite algebra $\A'$, where $\Omega'=\Omega-\cup_{i=1}^{k}\{\omega_{j(i)}\}$ and $\A'$ the algebra of subsets of $\Omega'$.

\, 

We recall that a Cevian (Ceva line) is a line starting from a vertex and connecting the opposite face. The following construction allows to use Cevians as webs for the probabilistic simplex.
This construction permits also a description in terms of a continuous one parameter groups and enables us to introduce the notion of derivation on the simplex.

\begin{prop}
Assume the manifold of probability distributions is defined as above. Then, there exists an $(n+1)-$web which is parallelizable and given by Cevians lines of the simplex.
\end{prop}
Before we start a proof  {\it stricto sensu } let us describe the construction. The construction above works for any dimension.

There are two important families of ``vectors''.
\begin{itemize}
\item The family $\{\vec y_{k}(\bp)\}_{k=1}^{n}$ of non independent ``vectors'' in $\mathbb{R}^{n}$ starting at the point  $\bp$ and ending at vertex $\be_{k}$. This family of vectors lie on the simplex:
\begin{equation}\label{E:Y1}
\vec y_{k}(\bp)= \vec e_{k}-\vec p,\ k=1,\dots,n
\end{equation}
\smallskip 
\item The family $\{\vec z_{k}(\bp)\}_{k=1}^n$ of non independent ``vectors'' in $\mathbb{R}^{n}$. This starts at the point $\bq_k$, indexed by the vector $\vec q_{k}$, where the Cevian intersected the $k$-face,  opposite to the vertex from which is is issued. These vectors lie on the simplex: %$k$-th face of the simplex and ends at the vertex $\be_{k}$. This lies on the simplex:
 \begin{equation}\label{E:Z1}
 \vec z_{k}=(1-p^{k})^{-1}\vec y_{k}(\bp)=\vec e_{k}- \vec q_{k},
\end{equation}
where the vector 
\begin{equation} \vec{q}_{k}=\sum_{i\ne k}q_{k}^{i}\vec{e}_{i},\text{ with } q_{k}^{i}= \frac{p^{i}}{1-p^{k}},
\end{equation}
corresponds to a conditional probability distribution $P[\ -\mid \Omega\setminus \{\omega_{k}\}]$, lying on the $k$-th face of the simplex. For the 2-dimensional case this is corresponds to the edge.
\end{itemize}

From the definition~\eqref{E:Y1} of  $\mathbf{y}_{k}(\mathbf{p})$ we can as well introduce the following generating systems of vector fields on the simplex:
 \begin{equation} \label{E:X1}
\vec x_{k}(\bp)=p^{k}\vec y_k,
\end{equation}

where this generating system verifies the condition
\begin{equation} \label{E:X2}
\sum_{k=1}^{n}\vec x_{k}(\bp)=0.
\end{equation}

\, 

A point $\mathbf{p}$ in the interior of the simplex can be expressed as:
 \begin{equation}\begin{aligned}
 \vec p &= p^{k}\vec e_{k}+ (1-p^{k})\sum_{i\ne k}q^{i}_{k}\vec e_{i}, \text{ with } 0<p^{k}<1,\\
 &=p^{k}\vec e_{k}+(1-p^{k})\vec q_{k},
 \end{aligned}
 \end{equation}
and this for {\it each } choice of the vertex $\vec e_{k}$.

\begin{center}
\begin{tikzpicture}[scale=1.5]
 \draw(0.1,0) node{$\mathbf{0}$};
 \draw [black,dashed] (0, 0) -- (0,3);
 \draw(0,3.1) node{$\mathbf{e}_{3}$};
\draw[black,dashed]  (0, 0) -- (3.5,-1.5);
\draw(3.5,-1.5) node{$\mathbf{e}_{2}$};
\draw [black,dashed] (0, 0) -- (-3.5,-0.5);
\draw(-3.5,-0.5) node{$\mathbf{e}_{1}$};
\draw [black] (0, 3) -- (-3.3,-0.46);
\draw [black] (0, 3) -- (3.3,-1.45);
\draw [black] (3.3, -1.45) -- (-3.3,-0.46);
\draw[black]  (0, 3) -- (-1.45,-0.75);
\draw[black,dashed]  (0, 0) -- (-1.5,-0.75);
\draw(-1.5,-0.85) node{$\mathbf{q}_{3}$};
\draw[black,dashed]  (0, 1.6) -- (-0.7,1.2);
\draw[black,very thick,dashed,->]  (0, 0) -- (-0.7,1.2);
\draw(-0.8,1.2) node{$\mathbf{p}$};
\draw(0.15,0.9) node{$p^{3}$};
\draw(0.2,2.3) node{$1-p^{3}$};
\draw(-0.6,0.6) node{$\vec p$};
\draw[black,very thick,->]  (-0.7,1.2) -- (-0.31,2.2);
\draw(-0.8,1.7) node{$\vec x_{3}(p)$};
\end{tikzpicture}
\end{center}

Then the independent vectors $ \vec x_{1}(\bp),\vec x_{2}(\bp),\dots, \vec x_{n}(\bp)$ define an invariant generating system under the permutation of the vertices. Moreover, it allows to define a barycentric coordinate system. Note that this barycentric coordinate system is just the coordinate system induced by the three Ceva lines intersecting at $\bp$.

Consider the statistical manifold defined for a discrete, finite sample space. This corresponds to an $n-$simplex. In the case of a manifold $S$ of multinomial distributions we can use  as coordinate systems $\theta=(\theta^1,\dots,\theta^n)$ such that:  

\[\theta^1=p_1,\dots,\,  \theta^n=p_n,\, \theta^{n+1}=p_{n+1}\]
and satisfying the condition \[\theta^1+\dots+\theta^{n+1}=1.\] 

As for the tangent space, it has $n$ generators $\partial_1,...,\partial_n$ 

where 
\[\partial_il(x,\theta)=\frac{\delta_{x_i}}{\theta^i}-\frac{\delta_{x_{n-1}}}{\theta^{n+1}}.\] 
 
In the $\theta^{k}=p^{k}, \ k=0, \dots, n$ coordinate system this space can be realized as an open simplex whose the vertices are given by $\mathbf{e}_{k}=\underbrace{(0,\dots,0,\underset{k}1,0,\dots,0)}_{d+1}$.

\begin{rem}
    An interesting remark stems from this change of coordinates. Let us define $(\eta^i)_{i=1}^{n+1}$ by \[\eta^1=2\sqrt{p_1},\dots, \eta^{n+1}=2\sqrt{p_{n+1}},\]
so that \[\sum_{i=1}^{n+1} (\eta^i)^2=4\] which is exactly the equation of a sphere of radius 2. This interesting change of coordinates presents the manifold as a part of the $n$-dimensional sphere with radius 2, embedded in the $(n+1)$-Euclidean space with coordinate system $(\eta^1,...,\eta^n)$. 
\end{rem}

Since we consider only those $(\theta)_{i=1}^n$ such that $\sum_{i=1}^n \theta_i=1$, we are thus investigating $n$ webs on a given $(n-1)$-face of the simplex. 

\,

The proof is by induction, presented in the geometric construction below.

\begin{proof}
Assume ${\bf e(1), e(2),... ,e(n+1)}$ is the set of vertices of the simplex. We consider the $(n-1)$-face corresponding to $\sum_{i=1}^n \theta_i=1$. Let us show that we can construct the $n$-web using the Cevians and that this forms a paralleilzable web. The   paralleilzability property follows from the fact that the connection of the webs is torsionless. In turn, the connection is torsionless because the manifold is dually flat in the sense of Amari. However, it is possible to show the parallelizability property directly using Ceva's theorem. 

\,

Consider the low dimensional case. Let $e(1), e(2),e(3)=\{A,B,C\}$. 
By Ceva's relation we have the following relation:

\begin{equation}\label{E:1}\frac{A'B}{A'C}\times\frac{B'C}{B'A}\times\frac{C'A}{C'B}=-1.\end{equation} 

This relations holds for a triple of lines $(AA'), (BB')$ and $(CC')$ which are either concurrent in one point $k$ or parallel.

\,

In the first case, one draws lines from each vertex to its opposite edge, being concurrent at one point $k$ and satisfying the relation~\ref{E:1}. The second, is obtained by drawing three parallel lines, where each line contains one unique vertex of the triangle and the relation~\ref{E:1} is satisfied.

It is easy to verify that the construction above of Cevians on the 2-face of a 3-simplex corresponds to the definition of a 3-web. Applying the theorem of Ceva, we can draw those Cevians as parallel lines.  This happens to be exactly the definition of a parallelizable webs.

\,

The construction to higher dimensions is done by induction. This essentially mimics the steps of the proof in \cite{buba}. As a result, we have $(n+1)$-webs given by Cevians on an $n$-dimensional face and those webs are are parallelizable.

Therefore, we have ended the proof. 

\end{proof}

\subsection{Optimal learning and coordinates on a honeycomb lattice}

\,

 Let $\sS$ be a statistical manifold of exponential type (we consider this for continuous/discrete sample spaces). This object takes its roots in~\cite{Lehman} and \cite{Ch}, namely Markov categories for the latter reference. 
 
 \, 
 
 Optimal learning curves exist on totally geodesic submanifolds of $\sS$ satisfying the WDVV equation (being a Frobenius manifold).

\,

The final statement of this section is now presented. 
\begin{thm}\label{P:Hex}
Let $\cD\subset N$ be a domain of a totally geodesic submanifold $N$ of $\sS$. Suppose  that it satisfies the Frobenius manifold requirements. 
Then $\cD$ is equipped with a hexagonal web. The learning on $\cD$ can thus be described using a local hexagonal lattice and is optimal. 
\end{thm}

\begin{proof}
This follows partly from  \cite[Sec.4]{CMM2023}. Having the Frobenius manifold condition implies that we have an integrable system. In particular, it is also endowed with the fact that the Chern connection is flat and that the torsion vanishes. 

A web $W(n + l,n,r)$ is said to be hexagonal if all $n\choose 3$ its 3-
subwebs are hexagonal.

By \cite{Ak} the web is hexagonal if and only if its curvature vanishes. Therefore, we have the conclusion. 

Therefore, locally, on such a domain  $\cD$ the learning process can be expressed using the hexagonal structures.

\end{proof}
\subsubsection{Conclusion}
The previous result brings  many new interesting possibilities to explore. 
One immediate application  could be to use the hexagonal webs as a local hexagonal lattice providing us with some local hexagonal coordinates. There are many properties one may take advantage of. For instance hexagonal webs are related to having group webs. Due to such properties, local considerations on the learning can be simplified using the symmetries: coming from the group webs on one side and coming from the dihedral symmetries of the hexagon. This will be the subject of a new paper.

\,

%Having a tesselation given by this hexagonal lattice simplifies considerations around our {\it statistical data} greatly. In particular, if the hexagons are regular, one can use the (regular) lattice to consider the learning process only in one hexagon, which plays the role of a fundamental domain and use the vertices of the lattice to introduce certain types of coordinates. 

%A hexagonal lattice can be used to define optimal learning (geodesic) curves.  

%\begin{cor}
%The optimal curves can be defined using a honeycomb lattice.   
%\end{cor}

\bibliographystyle{acm}%Used BibTeX style is unsrt
\bibliography{sample}

\end{document}